\documentclass[a4paper,12pt]{article}
\usepackage[ansinew]{inputenc}
\usepackage{amsfonts}
\usepackage{latexsym}
\usepackage{amsmath}
\usepackage{amssymb}
\usepackage{graphicx}
\usepackage{mathrsfs}
\usepackage{color}

\newcommand{\R}{\mathbb{R}}

\newcommand{\C}{\mathbb{C}}
\newcommand{\N}{\mathbb{N}}

\newtheorem{defin}{Definition}[section]

\newtheorem{theorem}[defin]{Theorem}

\newtheorem{exa}[defin]{Example}
\newenvironment{example}{\begin{exa}\rm}{\end{exa}}
\newtheorem{lemma}[defin]{Lemma}
\newtheorem{corollary}[defin]{Corollary}
\newenvironment{proof}
{\noindent{\it Proof.}}{\hfill $\Box$\par\vspace{2.5mm}}

\newtheorem{que}{Question}

\newtheorem{pro}{Problem}

\numberwithin{equation}{section}

\title{\bf\Large Representation of finite order solutions to linear differential equations with exponential sum coefficients}
\author{Xing-Yu Li, Jun Wang and Zhi-Tao Wen\footnote{Wen is the corresponding author}}
\date{}
\begin{document}

\maketitle
\begin{abstract}
We show a necessary and sufficient condition on the existence of finite order entire solutions of linear differential equations
    $$
    f^{(n)}+a_{n-1}f^{(n-1)}+\cdots+a_1f'+a_0f=0,\eqno(+)
    $$
where $a_i$ are exponential sums for $i=0,\ldots,n-1$
with all positive (or all negative) rational frequencies and constant coefficients. Moreover, under the condition that there exists a finite order solution of (+) with exponential sum coefficients having rational frequencies and constant coefficients, we give the precise form of all finite order solutions, which are exponential sums. It is a partial answer to
Gol'dberg-Ostrovski\v{i} Problem and Problem 5 in \cite{HITW2022} since exponential sums are of completely regular growth.

\bigskip


\end{abstract}

\section{Introduction}

The solutions of the linear differential equation
	\begin{equation}\label{lde}
      f^{(n)}+a_{n-1}(z)f^{(n-1)}+\cdots +a_1(z)f'+a_0(z)f=0
      \end{equation}
with entire coefficients $a_0(z),\ldots,a_{n-1}(z)$ are entire. To avoid ambiguity, we assume
that $a_0(z)\not\equiv 0$.
The following two classical theorems give general information concerning finite order solutions of \eqref{lde}.

\bigskip\noindent
\textbf{Wittich's theorem.} \cite[Theorem~4.1]{Laine} or \cite[p.~6]{Wittich}
\emph{The coefficients $a_0(z)$, $\ldots$, $a_{n-1}(z)$ of \eqref{lde} are polynomials if and only
if all solutions of \eqref{lde} are of finite order.}

\bigskip\noindent
\textbf{Frei's theorem.} \cite[Theorem~4.2]{Laine} or  \cite[p.~207]{Frei0}
\emph{Suppose that at least one coefficient in \eqref{lde} is transcendental, and that
$a_j(z)$ is the last transcendental coefficient, that is, the coefficients
$a_{j+1}(z),\ldots,a_{n-1}(z)$, if applicable, are polynomials. Then \eqref{lde} possesses at
most $j$ linearly independent solutions of finite order.}

\bigskip
Frei \cite{Frei} proved that the equation
	\begin{equation*}\label{Frei-eqn}
	f''+e^{-z}f'+\alpha f=0,\quad \alpha\in\C\setminus\{0\},
	\end{equation*}
has a subnormal solution, which is defined as
    $$
    \limsup_{r\to\infty}\frac{\log T(r,f)}{r}=0,
    $$
if and only if $\alpha=-m^2$ for a positive
integer $m$. The subnormal solution $f$ is a polynomial in $e^z$ of degree $m$, that is, an exponential sum of the form
	\begin{equation*}\label{Frei-f}
	f(z)=1+C_1e^z+\cdots+C_me^{mz}, \quad C_j\in \mathbb{C}\setminus\{0\}.
	\end{equation*}
The result above has been generalized by Ozawa \cite{Ozawa}, to the differential equation
    \begin{equation}\label{Ozawa-eqn}
    f''+e^{-z}f'+(az+b)f=0,\quad a,b\in\C.
    \end{equation}
He showed that there is no entire solution of finite order satisfying the differential
equation \eqref{Ozawa-eqn}, if $a\neq 0$.

Ozawa posed a question whether there exists a finite order solution of the equation
    \begin{equation}\label{Ozawa}
    f''+A(z)f'+B(z)f=0,
    \end{equation}
where $A=e^{-z}$ and $B$ is a non-constant polynomial. This question is
known as Ozawa's problem. This problem has been
answered partly by Amemiya-Ozawa \cite{AO} and
by Gundersen \cite{Gary}, while the complete solution is given by Langley \cite{Langley}. It was shown in \cite{Langley} that if $A=ce^{-z}$, where $c\in\C\setminus\{0\}$, and $B(z)$ is a non-constant polynomial, then all non-trivial solutions of \eqref{Ozawa} are infinite order.

The general problem of Ozawa has been proposed by
 Heittokangas, and given by Gundersen, see \cite[Section~5]{Garyproblem}.
More and more papers focus on the existence of finite order
solutions of second linear differential equations \eqref{Ozawa} with entire coefficients, even precise form of finite order solutions, see \cite{HILT1,HILT2,ProblemOzawa,Long,Wang-Laine,Wittich3}.

The examples above inspire us to focus on the finite order solutions of linear differential equations \eqref{lde} with exponential sum coefficients. An exponential sum is an entire
function of the form
    \begin{equation}\label{expsum.eq}
    f = F_1e^{\lambda_1z} + \cdots + F_me^{\lambda_mz},
    \end{equation}
where $\lambda_1,\ldots,\lambda_m\in\C$ are distinct constants called frequencies or leading coefficients of $f$,
and the coefficients $F_1,\ldots,F_m$ are polynomials. Observe that a polynomial is a special
case of an exponential sum. We also note that any exponential sum is a solution of an
equation of the form \eqref{lde} with polynomial coefficients, see \cite{VPT}.

The literature contains numerous examples of finite order solutions of linear differential equations of the
form \eqref{expsum.eq} with exponential sum coefficients, which is related to the
open problem 5 in \cite{HITW2022}, which asks whether finite order
transcendental solutions of \eqref{lde} are always of completely regular growth whenever the
coefficients of \eqref{lde} are exponential polynomials.

In \cite{WGH} the authors, including one of the authors of this paper, study linear differential equations with exponential polynomial coefficients, where exactly one coefficient
is of order greater than all the others. The main result shows that a nontrivial exponential polynomial
solution of such an equation has a certain dual relationship with the maximum order coefficient.

From these considerations, it is interesting for us to consider the finite order solutions
of differential equations
    \begin{equation}\label{expdiff.eq}
    f^{(n)}+A_{n-1}(z)f^{(n-1)}+\cdots+A_1(z)f'+A_0(z)f=0
    \end{equation}
with exponential sum coefficients $A_i(z)$ for $i=0,1,\ldots,n-1$ of the form \eqref{expsum.eq}, where the frequencies
$\lambda_1,\ldots,\lambda_m$ are rational and the coefficients $F_1,\ldots,F_m$ are constants.
 The problems from Ozawa and \cite[Problem 5]{HITW2022} inspire us to consider the following questions:
 \begin{itemize}
 \item[(1)]
 Can one find a condition which guarantees the existence of a non-trivial finite order solution of \eqref{expdiff.eq}?
 \item[(2)]
 Suppose that there exists a finite order solution $f$ of \eqref{expdiff.eq}, is it true that $f$ is of completely regular growth?
 \end{itemize}
 In this paper, we will give a necessary and sufficient condition on the existence of finite order entire solutions of linear differential equations \eqref{expdiff.eq} when the frequencies of exponential sum coefficients are either all positive or all negative. Moreover, we discuss the representation of finite order solutions of differential equations \eqref{expdiff.eq}, which are of completely regular growth. Several examples
are given to illustrate the results.

This paper is organized as follows. In Section \ref{complete}, we discuss the problem on completely regular growth. The theorems and examples are given in Section \ref{main.sec}. We recall the fundamental results on formal solutions of linear differential equations in Section \ref{power.sec}.
The proofs of the main results are given in Section \ref{proof.se}.

\section{Completely regular growth}\label{complete}

The Phragm\'en-Lindel\"of indicator function of an entire function $f(z)$ of finite order $\rho(f)>0$ is given by
    $$
    h_f(\theta)=\limsup_{r\to\infty}r^{-\rho(f)}\log |f(re^{i\theta})|
    $$
for $\theta\in [-\pi,\pi)$. We say that $f$ is of completely regular growth (see \cite[pp.~139-140]{Levin1} or \cite[pp.~6-8]{Ponkin}) if there exists a sequence of Euclidean discs $D(z_k,r_k)$ satisfying
	\begin{equation}\label{r.eq}
	\sum_{|z_k|\leq r}r_k=o(r)
	\end{equation}
such that
	\begin{equation}\label{crg}
	\log |f(re^{i\theta})|= (h_f(\theta)+o(1))r^{\rho(f)}, \quad re^{i\theta}\not\in\bigcup_k D(z_k,r_k),
	\end{equation}
as $r\to\infty$ uniformly in $\theta$. For example, a transcendental exponential polynomial function is of completely regular growth, see \cite[Lemma 1.3]{GOP}.

A set $E\subset\C$ which can be covered by a sequence of discs $D(z_k,r_k)$ satisfying \eqref{r.eq} is known as a $C_0$-set. The projection of a $C_0$-set on the positive real axis has zero upper linear density.

Petrenko \cite[pp.~104--112]{Petru} has shown that transcendental
solutions to linear differential equations \eqref{lde} with polynomial
coefficients are of completely regular growth. See also \cite{Stein}
for parallel discussions. Based on Petrenko's result,
Gol'dberg and Ostrovski\v{i} stated the following problem, see \cite[p.~300]{problembook}.

\bigskip
\noindent
\textbf{Gol'dberg-Ostrovski\v{i}'s Problem}.~\emph{Suppose that $f$ is a finite order transcendental solution of \eqref{lde}
whose coefficients are entire functions of completely regular growth. Is it true that $f$ is of completely regular growth?}

\bigskip

Without the condition of completely regular growth for the coefficients, the answer to this problem is negative, see \cite[p.~300]{problembook}.
In a recent paper, Bergweiler \cite{Bergweiler} has answered this question in the negative.
It is natural to consider linear differential equations with transcendental exponential
sum coefficients as an important special case of this problem, see \cite[Problem 5]{HITW2022}.
In Sections \ref{main.sec}, we discuss that the representation of
finite order solutions
of linear differential equations \eqref{lde} with exponential sum coefficients of the form \eqref{expsum.eq}, whose frequencies
$\lambda_1,\ldots,\lambda_m$ are rational and the coefficients $F_1,\ldots,F_m$ are constants, are exponential sums. It is still open whether all finite order solutions of \eqref{lde} with exponential polynomial coefficients are of completely regular growth.

\section{Results and examples}\label{main.sec}

In this paper, we focus on solutions of linear differential equations with exponential sum coefficients of the form \eqref{expsum.eq}. In particular, all frequencies $\lambda_1,\ldots,\lambda_m$ are rational and the coefficients $F_1,\ldots,F_m$ are constants. Then, in this case there exists a rational number $\lambda'$ such that $\lambda_j=t_j\lambda'$, where $t_j$ are integers for $i=1,\ldots,m$.
The discussion above shows us the linear differential equations
we considered about are of the form
	\begin{equation}\label{NPDE}
    e^{\gamma\lambda' z}f^{(n)}+P_{n-1}(e^{\lambda' z})f^{(n-1)}+\cdots +P_1(e^{\lambda' z})f'+P_0(e^{\lambda' z})f=0,
    \end{equation}
where $P_0,\ldots,P_{n-1}$ are polynomials and $\gamma\in\N_+\cup\{0\}$. We note here that the finite order solutions of second order differential equations \eqref{NPDE} are related to some special functions when $\gamma=0$, see \cite{Olver}.

For simplicity, we assume $\lambda'=1$. Hence, let us consider finite order solutions of linear differential equations
	\begin{equation}\label{NPDE1}
    e^{\gamma z}f^{(n)}+P_{n-1}(e^{z})f^{(n-1)}+\cdots +P_1(e^{z})f'+P_0(e^{z})f=0,
    \end{equation}
where $P_0,\ldots,P_{n-1}$ are polynomials and $\gamma\in\N_+\cup\{0\}$. Obviously, every solution of \eqref{NPDE1} is entire. Let us discuss finite order solutions of \eqref{NPDE1}. Now we state one of our main results as follows.

\begin{theorem}\label{gen.theorem}
Suppose that $\gamma=0$.
The linear differential equation \eqref{NPDE1} has a finite order entire solution $f$ of the form
    $$
    f(z)=e^{\lambda z}u(e^z),
    $$
where $\lambda\in\C$ and $u$ is a polynomial, if and only if $u$ is a polynomial solution of linear differential equations
    \begin{equation}\label{ration2.eq}
     \left(\begin{matrix}
    u, &tu', &  &\cdots, t^nu^{(n)}
    \end{matrix}\right)
\left(\begin{matrix}
	q_{0,0} & q_{0,1} \cdots & q_{0,n}\\
	\vdots &\vdots &\vdots \\
	q_{n,0} & q_{n,1} \cdots & q_{n,n}
\end{matrix}\right)
        \left(\begin{matrix}
    m_{0,0} & m_{0,1} \cdots & m_{0,n}\\
    \vdots &\vdots &\vdots \\
    m_{n,0} & m_{n,1} \cdots & m_{n,n}
    \end{matrix}\right)
    \left(\begin{matrix}
    P_0\\
    \vdots\\
    P_n
    \end{matrix}\right)=0,
    \end{equation}
where $P_n=1$ and $(m_{i,j})$ is a $(n+1)\times(n+1)$ matrix such that $m_{i,j}$ are the Stirling numbers of the second kind satisfying
    $$
    m_{i,j}=\frac{1}{i!}\sum_{k=0}^i(-1)^k\binom{i}{k}(i-k)^j,
    $$
and $(q_{i,j})$ is a $(n+1)\times(n+1)$ matrix such that
    $$
    q_{i,j}=\binom{j}{i} \lambda^{\underline{j-i}}\quad\text{for}\quad 0\leq i\leq j\leq n\quad \text{and}\quad q_{i,j}=0\quad\text{for} \quad 0\leq j<i\leq n
    $$
by denoting $\lambda^{\underline{n}}=\lambda(\lambda-1)\cdots(\lambda-n+1)$, and $\lambda$ is a root of the equation
 \begin{equation}\label{s.eq}
	\lambda^{n}+P_{n-1}(0)\lambda^{n-1}+
\cdots+P_1(0)\lambda+P_0(0)=0.
\end{equation}
\end{theorem}

\begin{theorem}\label{gen1.theorem}
Suppose that $\gamma=0$. If $f$ is a finite order solution of \eqref{NPDE1}, then $f$ is an exponential sum of the form
    \begin{equation}\label{arb.sol}
f(z) = C_1e^{\lambda_1 z}\sum_{p=0}^{K_1}F_{1,p}(z)u_{1,p}(e^z) + \cdots +C_me^{\lambda_m z}\sum_{p=0}^{K_m}F_{m,p}(z)u_{m,p}(e^z),
    \end{equation}
where $F_{j,p}(z)$ and $u_{j,p}(z)$ are polynomials such that $\deg{F_{j,p}(z)}=p$ for $p=0,\ldots,K_j$, and $\lambda_j$ are the roots of \eqref{s.eq} and $K_j\in\mathbb{N^+}\cup\{0\}$ such that $K_1+\cdots+K_m\leq n$, and $C_j\in\C$ for any given $j=1,\ldots,m$.
\end{theorem}

We note here that $u_{j,p}$ in \eqref{arb.sol} is a polynomial solution of linear differential equations whose coefficients are related to $\lambda_j$. Now, let us state some corollaries by using Theorems \ref{gen.theorem} and \ref{gen1.theorem} together as follows.

\begin{corollary}\label{two.coro}
Suppose that $\gamma=0$. The linear differential equation \eqref{NPDE1} has a finite order entire solution $f$ if and only if $u$ is a polynomial solution of linear differential equations \eqref{ration2.eq}.
\end{corollary}

\begin{corollary}\label{one.coro}
Suppose that $\gamma=0$. Then
the differential equation \eqref{NPDE1} is solved by a finite order entire solution $f$ provided
there exists $\lambda\in\C\setminus\{0\}$ such that
    $$
    \lambda^n+P_{n-1}(z)\lambda^{n-1}+\cdots+\lambda P_1(z)+P_0(z)=0
    $$
holds. Moreover, the differential equation \eqref{NPDE1} has a finite order solution of the form $f(z)=Ce^{\lambda z}$, where $C\in\C\setminus\{0\}$.
\end{corollary}

In the following, we show the representation of finite order solutions of \eqref{NPDE} if $\gamma\in\N_+\cup\{0\}$.

\begin{theorem}\label{gen.theorem.mod}
Every finite order solution of linear differential equation \eqref{NPDE} is an exponential sum.
\end{theorem}

The following preliminary examples illustrate theorem \ref{gen.theorem} and theorem \ref{gen1.theorem} for third order linear differential equations.

\begin{example}
The function $f(z)=e^{-\frac{4}{3}z}(1-7e^{z})$ solves the differential equation
    $$
    f'''+3e^zf''+\left(-\frac{4}{3}-2e^z\right)f'-\left(e^z-\frac{16}{27}\right)f=0.
    $$
Obviously a polynomial $u(t)=1-7t$ is a solution of the differential equation
    $$
    t^2u'''+t(3t-1)u''+(1-7t)u'+7u=0
    $$
when $t=e^z$, where $\lambda=-4/3$ is a root of
    $$
    \lambda^3-4/3\lambda+16/27=(\lambda+4/3)(\lambda-2/3)^2=0.
    $$
\end{example}

\begin{example}\cite[Example 9.2]{GS1994}
The entire function $f=e^{-iz}+e^{-z}$ is a solution
of the equation
    \begin{equation}\label{exmp1.eq}
    f'''+(e^z+1)f''+((1+i)e^z+1)f'+(ie^z+1)f=0.
    \end{equation}
Obviously, it is known that $\lambda=-i$ and $\lambda=-1$ are the solutions of
    $$
    \lambda^3+\lambda^2+\lambda+1=0.
    $$
In fact, $f_1=e^{-iz}$ and $f_2=e^{-z}$ of the form \eqref{arb.sol} are linearly independent solutions of \eqref{exmp1.eq}. From Frei's theorem,
there are at most two linearly independent solutions
of \eqref{exmp1.eq}. Hence, every finite order solution of
\eqref{exmp1.eq} is a linear combination of $f_1$ and $f_2$.
\end{example}

\begin{example}
Consider the differential equation
    \begin{equation}\label{example2.eq}
    f'''-(3+e^z)f''+(3+2e^z)f'-(1+e^z)f=0.
    \end{equation}
It is clear that $\lambda=1$ satisfies
    $$
    \lambda^3-(3+t)\lambda^2+(3+2t)\lambda-(1+t)=0.
    $$
Form Corollary \ref{one.coro}, there exists a finite order entire solution of \eqref{example2.eq}.
Obviously, it is known that $f_1=e^z$ and $f_2=ze^z$ are two linearly independent solutions of \eqref{example2.eq}. Moreover, from Frei's theorem, there exist at most two finite order linearly independent solutions of \eqref{example2.eq}.
\end{example}

The next example shows that there does not exist any finite order solution of a linear differential equation from Theorem \ref{gen.theorem}.

\begin{example}
Let us consider a second order differential equation
    \begin{equation}\label{ex.eq}
    f''+f'+e^zf=0.
    \end{equation}
By \cite[Corollary~1]{Gundersen1988} it follows that every solution $f\not\equiv 0$ of \eqref{ex.eq} has infinite order. In the following, we use different method, that is Theorem~\ref{gen.theorem}, to illustrate there exists no non-trivial finite order solution of \eqref{ex.eq}.

Now we set $t=e^z$ and $f(z(t))=v(t)$. Then $v$ satisfies
    $$
    tv''+2v'+v=0.
    $$
Let us choose $\lambda$ such that $\lambda$ is a solution of
    $
    \lambda^2+\lambda=0.
    $
Hence $\lambda=-1$ or $\lambda=0$. By setting $v=t^\lambda u$, we get
$u$ is a solution of
    $$
    tu''+u=0
    $$
or
    $$
    tu''+2u'+u=0.
    $$
In either case, $u$ is not a non-zero polynomial. Therefore, there does not any finite order solution of \eqref{ex.eq} by Theorem \ref{gen.theorem}.

\end{example}

\section{Formal solutions of linear differential equations}\label{power.sec}

In order to seek all possible finite order solutions of
\eqref{NPDE1}, let us recall the fundamental results on formal solutions of linear differential equations
    \begin{equation}\label{LEDR.eq}
    v^{(n)}+a_{n-1}v^{(n-1)}+\cdots+a_1v'+a_0v=0,
    \end{equation}
where $a_i$ are meromorphic functions for $i=0,1,\ldots,n-1$. The point $t=t_0$ is called an \emph{ordinary point} of \eqref{LEDR.eq} if $t=t_0$ is not a pole of any of $a_i$ for $i=0,1,\ldots,n-1$. All solutions of \eqref{LEDR.eq} are entire for every given $t\in\C$ is an ordinary point of \eqref{LEDR.eq}. We call
$t=t_0$ a \emph{singularity of the first kind} for \eqref{LEDR.eq} if $a_k$ has at most a pole of $n-k$-th order
in the domain $|t-t_0|<\delta$, where $k=0,1,\ldots,n-1$. Moreover, we say $t=t_0$ is a \emph{regular singular point} of \eqref{LEDR.eq} if all solutions do not grow faster than some inverse powers of $|t-t_0|$, as $t\to t_0$ in $|\arg (t-t_0)|\leq \pi$. Fuchs showed that
the point $t=t_0$ is a regular singular point for \eqref{LEDR.eq} if and only if it is a singularity of the first kind, see \cite{CL,Singer} for more details.
The point $t=t_0$ is called
 \emph{singularity of the second kind} for \eqref{LEDR.eq}
if there exists at least one term $a_k$ such that $a_{k}{(t-t_0)}^{n-k}$ is not analytic at $t_0$, where $0\leq k\leq n-1$.

Firstly, we focus on the regular singular point case. In order to seek for the formal solutions of \eqref{LEDR.eq}, let us consider the origin as a regular singular point. Then we assume that the differential equation \eqref{LEDR.eq} is of the form
    \begin{equation}\label{L.eq}
    t^nv^{(n)}+t^{n-1}b_1v^{(n-1)}+\cdots+b_nv=0,
    \end{equation}
 where $b_j$ are entire functions.
The \emph{indicial equation} associated with \eqref{L.eq} is
    \begin{equation}\label{indicial}
    Q(\lambda)=\lambda^{\underline{n}}+b_{1,0}\lambda^{\underline{n-1}}
    +\cdots+b_{n-1,0}\lambda+b_{n,0}=0,
    \end{equation}
where $\lambda^{\underline{n}}=\lambda(\lambda-1)\cdots(\lambda-n+1)$ and $b_j=\sum_{k=0}^\infty b_{j,k}t^k$ for $j=1,2,\ldots,n$.
Assuming that a set $E:=\{\lambda_1,\ldots,\lambda_n\}$ consists of all the roots of \eqref{indicial}, we divided the set $E$ into $m$ subsets $E_j~(1\leq j\leq m)$ such that the difference of each two elements is an integer in one subset and is a non-integer in distinct subsets. The number of elements in each set $E_j$ is denoted by $k_j$ for $j=1,\ldots,m$. It is clear that $k_1+\cdots+k_m=n$. We rename the elements in each subset as follows
    $$
    E_j:=\{\lambda_{j,1},\ldots,\lambda_{j,k_j}\}\quad\text{for}\quad j=1,2,\ldots,m.
    $$
Moreover, the elements in each set are ordered according to increasing moduli.

It is well known that there exist $k_j$ linearly independent solutions of \eqref{L.eq} which associate with the elements in each subset $E_j$.
The Frobenius method (e.g., see \cite[pp.~132-135]{CL}, \cite[pp.~396-402]{Ince}), or an alternative method of descending order of equations (see \cite[p.~364]{Ince}, \cite[Theorem 4.8.12]{Schafke-Schmidt}) shows that these $k_j$ linearly independent solutions $v_{j,\kappa}$ which associate with the roots $\lambda_{j,\kappa}$ are of the form
    \begin{equation}\label{ff.eq}
    v_{j,\kappa}=t^{\lambda_{j,1}}(\varphi_{j,\kappa,0}+\varphi_{j,\kappa,1}\log t+\cdots+\varphi_{j,\kappa,\kappa-1}(\log t)^{\kappa-1}),
    \end{equation}
where $\varphi_{j,\kappa,i}$ are entire functions for $\kappa=1,\ldots,k_j$ and $i=0,\ldots,\kappa-1$. In this way, we can find $n$ linearly independent solutions of \eqref{L.eq} $v_{1,1},\ldots,v_{1,k_1},\ldots,v_{m,1},\ldots,v_{m,k_m}$, where
$k_1+\cdots+k_m=n$.

For the general case, we focus on the differential equation of the form
    \begin{equation}\label{L2.eq}
    t^nv^{(n)}+t^{n-1}b_1v^{(n-1)}+\cdots+b_nv=0,
    \end{equation}
 where at least one of $b_j~(j=0,1,\ldots,n-1)$ is not entire.
From \cite[\S  15.23]{Ince}, the formal solution of \eqref{L2.eq} is
    \begin{equation}\label{ff2.eq}
    v_{j,\kappa}=t^{\lambda_{j}}(\varphi_{j,\kappa,0}+\varphi_{j,\kappa,1}\log t+\cdots+\varphi_{j,\kappa,\kappa-1}(\log t)^{\kappa-1}),
    \end{equation}
where at least one of $\varphi_{j,\kappa,i}$ is not entire with an essential singularity at 0 for $\kappa=1,\ldots,k_j$ and $i=0,\ldots,\kappa-1$, and $\lambda_j$
is a root of a $n$-th order algebraic equation.

\section{Proofs of Theorems}\label{proof.se}

In this section, we proceed to prove our theorems one by one.
\subsection{Proof of Theorem \ref{gen.theorem}}
Suppose that $f$ is a solution of \eqref{NPDE1}. Setting $t=e^z$ and $f(z(t))=v(t)$, we have
    $$
    \left(\begin{matrix}
    f, &\frac{df}{dz},& \cdots ,&,\frac{d^nf}{dz^n}
    \end{matrix}\right)=
     \left(\begin{matrix}
    v, & t\frac{dv}{dt} , &\cdots&,t^n\frac{d^nv}{dt^n}
    \end{matrix}\right)
    \left(\begin{matrix}
    m_{0,0} & m_{0,1} \cdots & m_{0,n}\\
    \vdots &\vdots &\vdots \\
    m_{n,0} & m_{n,1} \cdots & m_{n,n}
    \end{matrix}\right),
    $$
where $(m_{i,j})$ is a $(n+1)\times (n+1)$ matrix such that $m_{0,0}=1$, $m_{i,0}=0$, $m_{0,j}=0$ and
$m_{ij}=im_{i,j-1}+m_{i-1,j-1}$.
Then \eqref{NPDE1} is transformed into
    \begin{equation}\label{t3.eq}
     \left(\begin{matrix}
    v, &tv', &\cdots, &t^nv^{(n)}
    \end{matrix}\right)
        \left(\begin{matrix}
    m_{0,0} & m_{0,1} \cdots & m_{0,n}\\
    \vdots &\vdots &\vdots \\
    m_{n,0} & m_{n,1} \cdots & m_{n,n}
    \end{matrix}\right)
    \left(\begin{matrix}
    P_0\\
    \vdots\\
    P_n
    \end{matrix}\right)=0,
    \end{equation}
where $v'=dv/dt$ and $P_n=1$. It is easy for us to write \eqref{t3.eq} as
    \begin{equation}\label{alpha.eq}
    \alpha_n(t)t^nv^{(n)}+\alpha_{n-1}(t)t^{n-1}v^{(n-1)}+\cdots+
    \alpha_{0}(t)v=0,
    \end{equation}
where $\alpha_i(t)$ is a polynomial for $i=0,1,\ldots,n$ such that
    $$
    \alpha_i(t)=\sum_{j=0}^nm_{i,j}P_j(t).
    $$
Obviously, every singularity of \eqref{alpha.eq} is a regular singular point since $\alpha_n(t)=1$. According to Section \ref{power.sec}, we know that the indicial equation associated with \eqref{alpha.eq} is
    \begin{equation}\label{lambda.eq}
    \lambda^{\underline{n}}+\alpha_{n-1}(0)\lambda^{\underline{n-1}}
    +\cdots+\alpha_1(0)\lambda+\alpha_0(0)=0,
    \end{equation}
where $\lambda^{\underline{n}}=\lambda(\lambda-1)\cdots(\lambda-n+1)$. Keeping in mind that $m_{i,j}$ are the Stirling numbers of the second kind such that
    \begin{equation}\label{m.eq}
    \lambda^k=\sum_{i=1}^{k} m_{i,k}\lambda^{\underline{i}}
    \end{equation}
for any $k\in\N$. It shows us that
    $$
    \sum_{i=0}^n\alpha_i(0)\lambda^{\underline{i}}=\sum_{i=0}^n\sum_{j=0}^nm_{i,j}P_j(0)\lambda^{\underline{i}}
    =\sum_{j=0}^n\sum_{i=0}^nm_{i,j}P_j(0)\lambda^{\underline{i}}=\sum_{j=0}^nP_j(0)\sum_{i=0}^nm_{i,j}\lambda^{\underline{i}}
    =\sum_{j=0}^nP_j(0)\lambda^j.
    $$
We note that the last equality is from \eqref{m.eq} and $m_{i,j}=0$ for $i>j$. Therefore, the indicial equation associated with \eqref{alpha.eq} is
    $$
    \lambda^n+P_{n-1}(0)\lambda^{n-1}+\cdots+P_1(0)\lambda+P_0(0)=0.
    $$
We follow the idea of Fuchs, that is, substitute a series to the equation and proceeding to a determination of the coefficients, see \cite{Hille}. We claim that \eqref{t3.eq} has a formal solution
    \begin{equation}\label{form.eq}
    v(t)=t^\lambda\sum_{k=0}^\infty \gamma_kt^k=t^\lambda u(t),
    \end{equation}
where $\lambda\in\C$. It is obvious that $u$ is analytic in $t\in\C$ and $v$ is analytic in $t\in\C\setminus\{0\}$. Substituting \eqref{form.eq} into \eqref{t3.eq}, we have
$u$ is an entire solution of
    \begin{equation}\label{u.eq}
     \left(\begin{matrix}
    u, &tu' & , &\cdots, t^nu^{(n)}
    \end{matrix}\right)
    \left(\begin{matrix}
     q_{0,0} & q_{0,1} \cdots & q_{0,n}\\
    \vdots &\vdots &\vdots \\
    q_{n,0} & q_{n,1} \cdots & q_{n,n}
   \end{matrix}\right)
        \left(\begin{matrix}
    m_{0,0} & m_{0,1} \cdots & m_{0,n}\\
    \vdots &\vdots &\vdots \\
    m_{n,0} & m_{n,1} \cdots & m_{n,n}
    \end{matrix}\right)
    \left(\begin{matrix}
    P_0\\
    \vdots\\
    P_n
    \end{matrix}\right)=0,
    \end{equation}
where $(q_{i,j})$ is a $(n+1)\times(n+1)$ matrix such that $q_{i,j}=C_j^i\lambda^{\underline{j-i}}$ for $0
\leq i\leq j\leq n$ and $q_{i,j}=0$ for $0\leq j<i\leq n$.

In the following, we proceed to prove $u(t)$ is a polynomial if and only if $f(z)=v(t(z))$ in \eqref{form.eq} is of finite order, where $t=e^z$.
It is easy to see that $f(z)$ is of finite order when $u(t)$ is a polynomial. Now let us assume that $f(z)$ is a finite order entire solution of \eqref{NPDE1} when $\gamma=0$. From \cite[Page 388]{Helmrath-Nikolaus}, we obtain that every transcendental entire solution $u(t)$ of \eqref{u.eq} is of rational order $\rho(u)$ such that $0<\rho(u)<\infty$. It shows that $u(e^z)$ is of infinite order provided that
$u(t)$ is of positive order by P\'{o}lya's theorem \cite{Polya}, which is impossible. The only possibility is $u(t)$ is a polynomial when $f(z)$ is of finite order, where $t=e^z$.

\subsection{Proof of Theorem \ref{gen1.theorem}}

In this subsection, we use the same notations in the proof of Theorem \ref{gen.theorem}. According to \eqref{ff.eq}, we know that formal solutions $f_{j}$ of \eqref{NPDE1} are of the form
    \begin{equation}\label{fTheorem2.eq}
    f_{j}(z)=e^{\lambda_{j}z}(\mu_{j,0}(e^z)+\mu_{j,1}(e^z)\psi_1(z)+\cdots+\mu_{j,\nu_j}(e^z)\psi_{\nu_j}(z))
    \end{equation}
for $j=1,2,\ldots,n$, where $\mu_{j,i}$ are entire functions and $\psi_i$ are polynomials such that $\deg(\psi_i)=i$, and $\lambda_{j}$ is a root of \eqref{lambda.eq}.

We claim that
if $f$ is a finite order entire solution
of \eqref{NPDE1}, then $f$ is an exponential sum in the case of $\gamma=0$. Let us state two lemmas to prove our assertion.

\begin{lemma}\label{order.lemma}
Suppose that $\gamma=0$ and $f$ is a finite order transcendental entire solution of \eqref{NPDE1}. Then $f$ is of order $\rho(f)=1$ and mean type.
\end{lemma}

\begin{proof}
Let $f$ be a transcendental entire solution of \eqref{NPDE1}
and let $\nu(r)$ be its central index. By Wiman-Valiron theory, see \cite[Theorem 30]{Valiron} or \cite[Theorem 12]{Hayman}, let $F\subset\R_+$ be a set of finite logarithmic measure such that
    \begin{equation}\label{nu.eq}
    \frac{f^{(i)}(z)}{f(z)}=\left(\frac{\nu(r)}{z}\right)^i(1+o(1))
    \end{equation}
holds for $i=1,\ldots,n-1$ and for $r=|z|\not\in F$, $z$ being chosen as
    \begin{equation}\label{bigpoint.eq}
    |f(z)|>M(r,f)\nu(r,f)^{-1/4+\delta},
    \end{equation}
where $0<\delta<1/4$. Substituting \eqref{nu.eq} into \eqref{NPDE1}, we have
    \begin{equation*}\label{v.eq}
    \left(\frac{\nu(r)}{z}\right)^n+P_{n-1}(e^z)\left(\frac{\nu(r)}{z}\right)^{n-1}
    (1+o(1))+\cdots+P_0(e^z)(1+o(1))=0
    \end{equation*}
holds for $r=|z|\not\in F$. Therefore, it yields that
    \begin{equation}\label{v1.eq}
    \nu^n(r)+P_{n-1}(e^z) z\nu^{n-1}(r)(1+o(1))+\cdots+P_0(e^z) z^n(1+o(1))=0
    \end{equation}
for $r=|z|\not\in F$. If $f$ is of finite order $\rho(f)=\rho$, then $\nu(r)\leq r^{\rho+\varepsilon}$ for large $r$. It shows us that $\nu(r)$ is nondecreasing and at most polynomial growth.

We claim that $\nu(r)= Cr(1+o(1))$ holds as $r\not\in F$, where $C$ is a nonzero constant. Suppose that there exists a set with infinite logarithmic measure $E$ such that
    \begin{equation}\label{assume.eq}
    \nu(r_n)= Cr_n^{\rho_0}(1+o(1))
    \end{equation}
for any $r_n\in E$ with $r_n\to\infty$ as $n\to\infty$, where $\rho_0\neq 1$. It is well known that $|P_j(e^z)|$ are bounded as $z\to\infty$ for $j=0,1,\ldots,n-1$ when $z$ satisfying \eqref{bigpoint.eq} are in the left-half plane. Obviously, it is impossible that \eqref{v1.eq} and \eqref{assume.eq} hold together. The only possibility is that $z$ satisfying \eqref{bigpoint.eq} are in the right-half plane. We denote the index $\{d_1,d_2,\ldots,d_j\}\subset\{0,1,\ldots,n-1\}$ such that $\deg{P_{d_1}}=\deg{P_{d_2}}=\cdots=\deg{P_{d_j}}>\deg{P_k}$
for $k\not\in\{d_1,d_2,\ldots,d_j\}$. It yields that
   \begin{equation*}\label{v2.eq}
    \bigg|\sum_{d=d_1}^{d_j}P_d(e^z)z^{n-d}\nu(r)^d(1+o(1))\bigg|
    \leq
    \nu^n(r)+\sum_{d\not\in\{d_1,d_2,\ldots,d_j\}}P_d(e^z)z^{n-d}\nu(r)^d (1+o(1)),
    \end{equation*}
which is impossible if \eqref{assume.eq} holds.
Therefore, we have $\nu(r)= Cr(1+o(1))$ as $r\not\in F$, where $C$ is a nonzero constant. Then we conclude that
$\nu(r)= Cr(1+o(1))$ holds for large $r$ without any exceptional set from \cite{Helmrath-Nikolaus}. Hence, $f$ is of order 1 and mean type by $\nu(r)\leq M(r,f)$ and \cite[Theorem 6]{Hayman}.
\end{proof}

\begin{lemma}\label{finite.lemma}
Suppose that $g$ is an entire function of the form
    \begin{equation}\label{fu.eq}
    g=\mu_1(e^z)\psi_1(z)+\cdots+\mu_\nu(e^z)\psi_\nu(z),
    \end{equation}
where $\mu_i$ are entire functions and $\psi_i$ are polynomials such that $\deg(\psi_1)<\cdots<\deg (\psi_\nu)$.
If $g$ is of order $\rho(g)=1$ and mean type, then $\mu_j(z)$ are polynomials for $j=1,\ldots,\nu$.
\end{lemma}

\begin{proof}
We assume that $\mu_i$ are transcendental, and denote $\Delta^0 g(z)=g(z)$ and $\Delta^n g(z)=
\Delta^{n-1}g(z+2\pi i)-\Delta^{n-1}g(z)$ for $n\in\N^+$. If $g$ is of the form \eqref{fu.eq}, then
    $$
    \Delta g(z) =\mu_1(e^z)\psi_{1,1}(z)+\cdots+\mu_\nu(e^z)\psi_{\nu,1}(z),
    $$
where $\psi_{j,1}(z)=\Delta \psi_j(z)$ for $j=1,\ldots,\nu$. We repeat it $\tau=\deg \psi_\nu$ finitely many times,
    \begin{equation}\label{doit.eq}
    \Delta^\tau g(z) =c \mu_\nu(e^z),
    \end{equation}
where $c$ is a constant. Therefore, $g$ is of order $\rho(g)=1$ and mean type implies that $\mu_\nu(e^z)$ is of order $\rho(\mu_\nu(e^z))\leq1$ and mean type or minimal type by \cite[Theorem 2.1]{Chiang-Feng}. Let us set
    $$
    g_1(z)=g(z)-\psi_\nu(z)\mu_\nu(e^z).
    $$
We repeat the argument above, having $g_1$ is of order
$\rho(g_1)\leq 1$ and at most mean type implies that $\mu_{\nu-1}(e^z)$ is of order
$\rho(\mu_{\nu-1}(e^z))\leq 1$ and at most mean type. By repeating the discussion above for at most finitely many times, it follows that $\mu_j(e^z)$ are of order
$\rho(\mu_j(e^z))\leq 1$ and at most mean type for $j=1,\ldots,\nu$. Using P\'olya's Theorem \cite{Polya}, we deduce that $\mu_j(z)$ is of zero order and at most mean type for $j=1,\ldots,\nu$.
We obtain a contradiction from \cite[Theorem 1.46]{Yang-Yi}.
Therefore, we prove Lemma \ref{finite.lemma}.
\end{proof}

It follows that $\mu_{j,i}$ in \eqref{fTheorem2.eq} are polynomials from Lemmas \ref{order.lemma} and \ref{finite.lemma}. Therefore, we prove our assertion of Theorem \ref{gen1.theorem}. In addition, every finite order solution of \eqref{NPDE1} is an exponential sum when $\gamma=0$.

\subsection{Proof of Corollary \ref{two.coro}}

Suppose that $\gamma=0$. From the proof of Theorem \ref{gen1.theorem}, we know that some finite order solutions of \eqref{NPDE1} are of the form
    $$
    f_{j}(z)=e^{\lambda_{j}z}(\mu_{j,0}(e^z)+\mu_{j,1}(e^z)\psi_1(z)+\cdots+\mu_{j,\nu_j}(e^z)\psi_{\nu_j}(z))
    $$
for $j=1,2,\ldots,m (m\leq n)$, where $\mu_{j,i}(z)$ and $\psi_i$ are polynomials such that $\deg(\psi_i)=i$ for $i=0,1,\ldots,\nu_j$, and $\lambda_{j}$ is a root of \eqref{lambda.eq}. Choosing $g(z)=f_j(z)e^{-\lambda_jz}$ which is of the form \eqref{fu.eq}, similar as \eqref{doit.eq}, we get
    $$
    \Delta^{\nu_j}g(z)=c\mu_{j,\nu_j}(e^z),
    $$
where $c$ is a constant. Since $f(z+2\pi i)$ is a finite order solution of \eqref{NPDE1},
it shows us that $e^{\lambda_j z}\mu_{j,\nu_j}(e^z)$ is a finite order solution of \eqref{NPDE1}.
According to Theorem \ref{gen.theorem}, it yields that $u_{j,\nu_j}(t)$ is a polynomial solution of \eqref{ration2.eq}.

\subsection{Proof of Theorem \ref{gen.theorem.mod}}

 Suppose that $f$ is a finite solution of \eqref{NPDE1} and $\gamma\in\N_+$. We assume $\gamma=1$. Obviously, all the coefficients of differential equations \eqref{NPDE1} are simply periodic functions with period $2\pi i$. From the same proof of Lemma~\ref{order.lemma}, we have $f$ is of order 1 and mean type. According to Ince's book \cite[\S 15.7]{Ince}, we can find $n$ linearly independent formal solutions of \eqref{NPDE1} of the form
$$
f(z)=e^{\tilde\lambda z}(\tilde\mu_1(z)\psi_1(z)+\cdots+\tilde\mu_\nu(z)\psi_\nu(z)),
$$
where $\tilde\mu_i(z)$ are periodic functions with period $2\pi i$ and $\psi_i(z)$ are polynomials such that $\deg(\psi_1)<\cdots<\deg (\psi_\nu)$, and $\tilde{\lambda}$ is a root of a $n$-th order algebraic equation. Now we denote $\Delta^0 f(z)=f(z)$ and $\Delta^n f(z)=
	\Delta^{n-1}f(z+2\pi i)-\Delta^{n-1}f(z)$ for $n\in\N_+$.
Using the same method in the proof of Lemma \ref{finite.lemma}, we get the equality
	$$
	\Delta^\tau (f(z)e^{-\tilde\lambda z})=c \tilde\mu_\nu(z),
	$$
where $c$ is a constant and $\tau$ is a positive integer. It follows that $\tilde\mu_\nu(z)$ is entire as well as $f(z)$. At the same time, it implies that $\tilde\mu_\nu(z)$ is of order $\rho(\tilde\mu_\nu(z))\leq 1$ and mean type or minimal type since
$f$ is of order $\rho(f)=1$ and mean type, and \cite[Theorem 2.1]{Chiang-Feng}. Let us set
	$$
	f_1(z)=f(z)e^{-\tilde\lambda z}-\psi_\nu(z)\tilde\mu_\nu(z).
	$$
We repeat the argument above, and obtain that $\tilde\mu_{\nu-1}(z)$ is of order $\rho(\tilde\mu_{\nu-1}(z))\leq 1$ and at most mean type by using $f_1$ is entire of order
	$\rho(f_1)\leq 1$ and at most mean type. By repeating the discussion above for at most finitely many times, it follows that $\tilde\mu_j(z)$ are entire of order $\rho(\tilde\mu_j(z))\leq 1$ and at most mean type for $j=1,\ldots,\nu$. Together with the result in Gross' book \cite[Theorem 2.12]{Gross}, it shows that $\tilde\mu_j(z)$ are of the form
	\begin{equation*}
		\tilde\mu_j(z)=\sum_{k=-n_{j}}^{n_{j}}c_{jk}e^{kz},
	\end{equation*}
where $c_{jk}$ are constants. Therefore, we know every finite order solution is an exponential sum.

\section{Acknowledgments}

The authors would thank the reviewers' great comments to improve the paper. Jun Wang is supported by the National Natural Science Foundation of China (No.~12471072). Zhi-Tao Wen is supported  by the National Natural Science Foundation of China (No.~12471076) and LKSF STU-GTIIT Joint-research Grant (No. 2024LKSFG06).

\bigskip
\noindent
\emph{X.-Y.~Li}\\
\textsc{Shantou University, Department of Mathematics,\\
Daxue Road No.~243, Shantou 515063, China}\\
\texttt{e-mail:19xyli@stu.edu.cn}

\bigskip
\noindent
\emph{J. Wang}\\
\textsc{School of Mathematical Sciences, Fudan University, \\
Shanghai 200433, P. R. China}\\
\texttt{e-mail:majwang@fudan.edu.cn}

\bigskip
\noindent
\emph{Z.-T.~Wen}\\
\textsc{Shantou University, Department of Mathematics,\\
Daxue Road No.~243, Shantou 515063, China}\\
\texttt{e-mail:zhtwen@stu.edu.cn}
\end{document}